\documentclass[12pt]{article}

\usepackage[utf8]{inputenc}
\usepackage[english]{babel}

\usepackage{amsthm}
\usepackage{amsmath}
\usepackage{amsfonts}
\usepackage{tikz}
\usepackage{appendix}
\usepackage{enumerate}
\usepackage[normalem]{ulem}
\usepackage{mathtools}
\usepackage{stmaryrd}
\usepackage[noadjust]{cite}
\usepackage{cases}
\usepackage{caption}

\usepackage[left=2.5cm, right=2.5cm, top=2.5cm, bottom=2.5cm]{geometry}

\usepackage[shortlabels]{enumitem}

\usepackage[hidelinks, bookmarks, bookmarksnumbered, pdfstartview={XYZ null null 1.00}]{hyperref}

\newtheorem{theorem}{Theorem}[section]
\newtheorem{lemma}[theorem]{Lemma}
\newtheorem{proposition}[theorem]{Proposition}

\theoremstyle{definition}
\newtheorem{definition}[theorem]{Definition}
\newtheorem{remark}[theorem]{Remark}

\DeclareMathOperator{\supp}{supp}
\DeclarePairedDelimiter{\abs}{\lvert}{\rvert}
\DeclareMathOperator{\rank}{rank}

\begin{document}

\title{A corona theorem for an algebra of Radon measures with an application to exact controllability for linear controlled delayed difference equations}

\author{S\'ebastien Fueyo\thanks{School of Electrical Engineering, 
   Tel Aviv University, Ramat Aviv 69978, Israel.} \and Yacine Chitour\thanks{Universit\'e Paris-Saclay, CNRS, CentraleSup\'elec, Laboratoire des signaux et syst\`emes, 91190, Gif-sur-Yvette, France.}}

\maketitle

\begin{abstract} 
This paper proves a corona theorem for the algebra of Radon measures compactly supported in $\mathbb{R}_-$ and this result is applied to provide a necessary and sufficient Hautus--type frequency criterion for the $L^1$ exact controllability of linear controlled delayed difference equations (LCDDE). Hereby, it solves an open question raised in \cite{chitour:hal-03827918}.
\end{abstract}



\section{Introduction}

Corona problems are relevant in linear infinite-dimensional control theory, especially for delay equations see \cite{YamamotoWillems,chitour:hal-03827918}. Exact controllability in finite time is often characterized in terms of a B\'ezout identity over appropriate functional algebras and hence obtaining an exact controllability criterion is tantamount to the resolution of a corona problem for measures or distributions compactly supported algebras.

Since the resolution of the corona problem in one dimension for holomorphic bounded functions in the unit disk by the celebrated paper \cite{Carleson1962}, the corona problem received a large attention. Carleson's result has been extended in various way, as for more general domains or algebras, see for instance a matrix version in the polydisk \cite{treil2005matrix} or in a multiply connected domains \cite{brudnyi2000matrix}, for some functions algebra on planar domains \cite{mortini2014corona} or for the algebra of almost periodic function with a Bohr--Fourier negatively supported \cite{frentz2014subalgebra}. The most closely corona theorem  related to the controllability of difference delay equations is stated in \cite[Corollary 3.3]{maad2011generators} for distributions positively compactly supported, but at the current state of the literature, it does not apply directly 
to the exact controllability of linear controlled delayed difference equations (LCDDE).


In this paper, we establish two results. The first one consists in the resolution of a corona theorem for a subalgebra of $M(\mathbb{R}_-)$, the commutative Banach algebra made of Radon measures compactly supported in $\mathbb{R}_-$. More precisely, for a finite number of $f_1,...,f_N$, each of them being a finite sum of Dirac measures supported in $\mathbb{R}_-$, we give a necessary and sufficient condition on the Laplace transform of the measures $f_1,...,f_N$ to obtain the existence of $g_1,...,g_N \in M(\mathbb{R}_-)$ such that
\begin{equation}
\label{eq_intro}
    f_1*g_1+...+f_N*g_N=\delta_0,
\end{equation}
where $*$ denotes the convolution product and $\delta_0$ the Dirac distribution at zero.
That result is then used to derive an $L^1$ exact controllability criterion (in finite time) for LCDDE expressed in the frequency domain, thus solving an open question raised in \cite{chitour:hal-03827918}. We emphasize that LCDDE can sometimes be used to address some control theoretic questions for 1-D hyperbolic partial differential equations \cite{CoNg,baratchart}.

The strategy of proof for the corona problem goes as follows : in a first step, we reduce the corona problem~\eqref{eq_intro} to a corona problem in a quotient Banach algebra. The second step goes by contradiction and relies on  Gelfand representation theory characterizing maximal ideal as the kernel of homomorphisms, in the spirit of \cite{frentz2014subalgebra,bottcher1999corona}. It is not immediate how to deduce our corona theorem from these references and we include a proof of it for sake of clarity (yet very similar to that of \cite{frentz2014subalgebra}). 
As for our second main result, it answers a question raised in \cite{chitour:hal-03827918}
where the sufficiency of a frequency domain criterium for $L^1$ exact controllability of a LCDDE was reduced to establishing the  corona theorem established previously. 


\section{Prerequisites and definitions}

We introduce the notations and the distributional framework needed in this article.

\subsection{Notations}
\label{sec_not}

In this paper, we denote by $\mathbb{N}$ and $\mathbb{N}^*$ the sets of nonnegative and 
positive integers, respectively. The set $\left\{1,\dots,N \right\}$ is denoted by $\llbracket 1,N\rrbracket$ for any 
$N \in \mathbb{N}^*$. We use $\mathbb{R}$, $\mathbb{R}_+=[0,+\infty)$, $\mathbb{R}_+^*$, $\mathbb{R}_-=(-\infty,0]$ and $\mathbb{C}$ to denote the sets of real numbers, 
nonnegative, positive, nonpositive real numbers and complex numbers respectively. For $s \in \mathbb{C}$, $\Re(s)$ and $\Im(s)$ denote the real and imaginary part of $s$, respectively. 

\subsection{Radon measures framework}
\label{sec_dis}
We give the Radon measures spaces that we use in this paper and for further details see for instance \cite[Section 2]{chitour:hal-03827918}. Denote $C_0(\mathbb{R})$ and $C_0(\mathbb{R}_+)$ the Fréchet spaces of continuous functions with the topology induced by the uniform convergence on compact sets on $\mathbb{R}$ and $\mathbb{R}_+$ respectively. The (topological) support of a function $\phi \in C_0(\mathbb{R})$ is the closure of the set $\left\{x \in \mathbb{R} |\, f(x) \neq 0 \right\}$. We note by $M(\mathbb{R}_-)$ and $M_+(\mathbb{R})$ the spaces of Radon measures defined on $\mathbb{R}$ with compact support included in $\mathbb{R}_-$ and bounded on the left respectively. The support of a Radon measure $\alpha \in M_+(\mathbb{R}) $, denoted $\supp( \alpha)$, is the complement of the largest open set on which $\alpha$ is zero. 
We note $\delta_{\lambda} \in M(\mathbb{R}_-)$ the Dirac distribution at $\lambda \in \mathbb{R}_-$. Endowed with the convolution $*$, the two spaces  $M_+(\mathbb{R})$ and $M(\mathbb{R}_-)$ become commutative unital algebras where the unit is $\delta_0$. 

For $T>0$, we denote by $\Omega_{-}^{T}$ the subspace  of $M(\mathbb{R}_-)$ 
made of the elements $h\in M(\mathbb{R}_-)$ of the form
\begin{equation}
\label{eq:omega}
    h=\sum_{j=0}^{N} h_j \delta_{-\lambda_j},\, \lambda_j\in [0,T],\,   h_j \in \mathbb{R},\,N \in \mathbb{N},
\end{equation}
where we assumed with no loss of generality that $\lambda_i \neq \lambda_j$ when $i \neq j$.
We introduce the subalgebra $\Omega^{\mathrm{bd}}_-:= \underset{T \in \mathbb{R}_+}{\cup} \Omega_{-}^{T}$ of $M(\mathbb{R}_-)$. 
 We define the (bilateral) Laplace transform in the complex plane $\mathbb{C}$ for $ \mu \in M_+(\mathbb{R})$ as
\begin{equation}
\label{laplace_transform_radon_measure}
\widehat{\mu}(s)=\int_{-\infty}^{+\infty} d\mu(t)e^{-st}, \quad s \in \mathbb{C},
\end{equation}
provided that 
the 
integral 
exists. We have $\widehat{\mu * \rho}(s)=\widehat{\mu }(s) \widehat{ \rho}(s)$, for all $\mu , \rho \in M_+(\mathbb{R})$ and $s \in \mathbb{C}$. For all $\lambda \in \mathbb{R}$, $e^{s \lambda}$ is the Laplace transform of the element $\delta_{-\lambda}$ in $s \in \mathbb{C}$. For an element $\mu \in M(\mathbb{R}_-) $, the Laplace transform reads:
\begin{equation}
    \widehat{\mu}(s)=\int_{-\infty}^{0} d\mu(t)e^{-st}, \quad s \in \mathbb{C},
\end{equation}
where the previous integral is understood as a Lebesgue integral on $(-\infty,0]$.

\subsection{The truncation operator}

We now define the truncation to positive times of a measurable function $f$ defined on $\mathbb{R}$ as the following mapping $\pi$ satisfying the equation
\begin{equation*}
  \left(\pi f \right)(t) = \begin{cases}
   f(t),\mbox{ if $t \ge 0$},\\
   0, \mbox{ if $t< 0$}.
 \end{cases}
\end{equation*}
Let us introduce the space $C_{0,+}(\mathbb{R})$ the space of continuous functions with support bounded on the left. The following properties of the truncation operator can be easily proved.


\begin{lemma}
\label{lem:sumlemmasYamamoto}
The following assertions hold true:
\begin{enumerate}[i)]
\item  \label{lem:A_2_Yamamoto}
For $ \alpha \in C_{0,+}(\mathbb{R})$, we have $\pi(\alpha)=0$ if and only if $\supp(\alpha) \subset (-\infty,0]$.
\item \label{lem:A_3_Yamamoto}
$\pi(\alpha*\beta)=\pi(\alpha* \pi \beta) $ for every $\alpha \in M(\mathbb{R}_{-})$ and $\beta \in C_{0,+}(\mathbb{R})$.
\end{enumerate}
\end{lemma}

\section{Topological properties of the quotient algebra $M(\mathbb{R}_-)/(p)$}

The aim of this section is to study the topological structure of the quotient algebra $M(\mathbb{R}_-)/(p)$ where $(p)$ is the principal ideal generated by any $p \in \Omega^T_- \subset M(\mathbb{R}_-)$, for some $T>0$, with the assumption that $p\neq 0$ and the support of $p$ is not reduced to the singleton $\{0\}$. In another words, $p$ is given by
\begin{equation}
\label{eq:omega}
    p=\sum_{j=0}^{N} p_j \delta_{-\lambda_j},\, \lambda_j\in [0,T],\,   p_j \in \mathbb{R},\,N \in \mathbb{N},
\end{equation}
and there exists $j \in \{0,...,N\}$ such that $p_j\lambda_j\neq 0$. 

We next recall the framework developed by Y. Yamamoto in \cite{yamamoto1989reachability}. Consider the bilinear form $( \cdot, \cdot )$ on $M(\mathbb{R}_-) \times C_0(\mathbb{R}_+)$ defined by
$$ ( w, \gamma ) :=(w*\pi \gamma) (0)=\int_{- \infty}^{0} dw(\tau) \gamma(-\tau),\quad w \in M(\mathbb{R}_-),\quad \gamma \in C_0(\mathbb{R}_+).$$

The space of  Radon measures $M(\mathbb{R}_-)$ is a normed algebra with the total variation norm
$$\|w\|_{\rm TV}:= \underset{\substack{\|\gamma\|_{\infty} \le 1, \\   \gamma \in C_0(\mathbb{R}_+)}}         
        {\sup} \abs*{( w,\gamma )},\quad w \in M(\mathbb{R}_-),$$
where $\|\gamma\|_{\infty}:= \underset{t \in \mathbb{R_+}}{\sup}|\gamma(t)|$ for $\gamma \in C_0(\mathbb{R}_+)$. We define
$$X^p:= \left\{ \gamma \in C_0(\mathbb{R}_+), \, \pi(p* \pi\gamma)=0 \right\}. $$
and we introduce the orthogonal complement of $X^p$
$$\left(X^p \right)^{\perp}:= \left\{w \in M(\mathbb{R}_-),\, ( w,\gamma )=0,\, \forall \gamma \in X^p \right\}.$$
One can see from the definition of the orthogonal complement that $\left(X^p \right)^{\perp}$ is a closed subspace of $M(\mathbb{R}_-)$. 
Thus we can define the normed quotient space $M(\mathbb{R}_-)/\left(X^p \right)^{\perp}$, see for instance \cite[Proposition 3.1, (ii), Chap. 3]{swartz2009elementary}, endowed with the norm
\begin{equation}
    \label{eq:norm_quotient}\|[w]\| :=\underset{\gamma \in \left(X^p \right)^{\perp}}{ \mathrm{inf}}\|w+\gamma \|_{\rm TV},\quad [w] \in M(\mathbb{R}_-)/\left(X^p \right)^{\perp},
    \end{equation}
where $\left[ w \right] \in M(\mathbb{R}_-)/\left(X^p \right)^{\perp}$ denotes any class of equivalence of  $M(\mathbb{R}_-)/\left(X^p \right)^{\perp}$.
 We denote by $(p):=\left\{p* \psi|\,\psi \in M(\mathbb{R}_-)\right\}$ the two-sided ideal generated by $p$ over the commutative algebra $M(\mathbb{R}_-)$. It turns out that the orthogonal complement of $X^p$ is in fact $(p)$ and we give a proof of that similar in the spirit of \cite[Lemma 2.18]{yamamoto1989reachability}.

\begin{lemma}
\label{lemm1}
    The following equation holds,
$$\left(X^p \right)^{\perp}=(p) .$$
\end{lemma}

\begin{proof}

Pick $p*\psi \in (p)$ with $\psi \in M(\mathbb{R}_-)$. For all $\gamma \in X^p$, we have
\begin{equation}
    ( p*\psi, \gamma )=(\psi*p*\pi \gamma)(0)=0,
\end{equation}
because $\gamma \in X^p$ implies that  $p*\pi \gamma(t)=0$ for $t \ge 0
$. Thus $(p) \subseteq\left(X^p \right)^{\perp} $.
Conversely, let $w \in \left(X^p \right)^{\perp} $. Take any $\phi \in  \mathcal{D}(\mathbb{R}_-)$, the space of smooth functions defined on $\mathbb{R}$ with compact support included in $\mathbb{R}_-$. A Neumann series argument proves that $p$ is invertible in $M_+(\mathbb{R})$ with respect to the convolution, and we denote $p^{-1} \in M_+(\mathbb{R}) $ its inverse. From Item~\ref{lem:A_3_Yamamoto} in Lemma~\ref{lem:sumlemmasYamamoto}, we have that the function $t\in \mathbb{R}_+ \mapsto \gamma(t):=\pi (p^{-1}* \phi)(t) $ belongs to $X^p$ and
\begin{equation}
\label{eq:mystere}
\begin{aligned}
    \pi(w*p^{-1}*\phi)(0)=\pi(w*\pi(p^{-1}*\phi))(0)
    =( w, \gamma)
    =0,
    \end{aligned}
\end{equation}
because $w \in \left(X^p \right)^{\perp} $. If we take, for all $t \in \mathbb{R}_+$, $\delta_{-t}*\phi$ instead of $\phi$ in Equation~\eqref{eq:mystere}, we get that 
\begin{equation}
\label{eq:lemm1}
\begin{aligned}
   \pi(w*p^{-1}*\phi)(t) &= \pi(w*p^{-1}*\delta_{-t}*\phi)(0)
   =0.
   \end{aligned}
\end{equation}
 From \eqref{eq:lemm1}, we have that $\pi(w*p^{-1}*\phi)$ is zero so that Item~\ref{lem:A_2_Yamamoto} in Lemma~\ref{lem:sumlemmasYamamoto} implies that the support of $w*p^{-1}*\phi \in C_{0,+}(\mathbb{R})$ is included in $(-\infty,0]$. Since it holds for any $\phi \in \mathcal{D}(\mathbb{R}_-)$, we have that $w * p^{-1}$ lies in $ M(\mathbb{R}_-)$. In particular, there exists $\psi \in M(\mathbb{R}_-)$ such that $w=p*\psi$. We deduce that $ \left(X^p \right)^{\perp}\subseteq (p) $, achieving the proof of the lemma.

\end{proof}

Thanks to Lemma~\ref{lemm1}, we have that the quotient normed space $M(\mathbb{R}_-)/\left(X^p \right)^{\perp}$ is in fact the normed quotient algebra equal to $M(\mathbb{R}_-)/(p)$ with unit $[\delta_0]$, see for instance \cite[Sect 1.4, Lemma 1.4.4]{kaniuth2009course} for a reference on normed quotient algebras. In particular, we have that $[w_1*w_2]=[w_1]*[w_2]$ and $[w_1+w_2]=[w_1]+[w_2]$ for all $w_1,w_2 \in M(\mathbb{R}_-)/(p)$. Our next step is to derive the following properties for the quotient algebra $M(\mathbb{R}_{-})/(p)$, which are a specification of \cite[Lemma 2.21]{yamamoto1989reachability} in the framework of our article.

\begin{theorem}\label{th3}
 The quotient algebra $M(\mathbb{R}_{-})/(p)$ is a commutative unital Banach algebra with $[\delta_0]$ as unit. Furthermore, we have
  \begin{equation}
    \label{def_norm}
          \|[w] \|=\underset{\substack{\|\gamma\|_{[0,T]} \le 1, \\  \gamma \in X^p}}  
         {\sup} \abs*{( w,\gamma )},\quad [w] \in M(\mathbb{R}_-)/(p).
              \end{equation}
\end{theorem}
\begin{proof}
We already know that $M(\mathbb{R}_{-})/(p)$ is a commutative unital algebra with unit $[\delta_0]$. It remains to prove that it is a Banach algebra. We have that $X^p \subset C_0(\mathbb{R}_+)$ with the topology induced by the uniform convergence on compact sets. By the definition of $p \neq 0$, we have that $\gamma \in X^p$ if and only if $\gamma \in C_0(\mathbb{R}_+)$ and it satisfies the difference delay equation
\begin{equation}
\label{diff_delay_equation}
     \sum_{j=0}^{N} p_j\gamma(t+\lambda_j)=0,\quad t \ge 0,
\end{equation}
where $p_j\lambda_j\neq 0$ for some $j \in \{0,...,N\}$. Thanks to Equation~\eqref{diff_delay_equation}, we have that the values on $\mathbb{R}$ of the function $\gamma$ are entirely constrained by the value of $\gamma$ on the interval $[0,T]$. Thus, the topology on $X^p$ is equivalent to the topology induced by the uniform convergence on the interval $[0,T]$. Therefore $X^p$ is a Banach space endowed with the norm $\|\phi \|_{[0,T]}= \underset{t \in [0,T]}{\sup} \abs{\phi(t)} $ with $\phi \in X^p$. We denote by $(X^p)'$ the topological dual of $X^p$, i.e. the space of continuous linear forms on $X^p$ with respect to the topology induced by the norm $\|\cdot\|_{[0,T]}$. We have that the space $(X^p)'$ is a Banach space endowed with the norm $$\|x\|_{\left(X^p\right)'}:=\underset{\substack{\|\phi\|_{[0,T]} \le 1, \\  \phi \in X^p}}         
        {\sup} \abs*{\langle x,\phi \rangle_{X^p}},\quad x \in \left(X^p\right)',$$
        where $\langle \cdot, \cdot \rangle_{X^p}$ denotes the duality product on $X^p$.
We define the linear map
 \begin{equation}
 \begin{aligned}
   h:\,  \quad M(\mathbb{R}_-)/\left(X^p \right)^{\perp} 
   \rightarrow   (X^p)' \\
 [w] \mapsto \left(\phi \in X^p \mapsto ( w,\phi )\right).
 \end{aligned}
 \end{equation}

We claim that the linear map $h$ is well-defined and is an isometric isomorphism between  $M(\mathbb{R}_-)/\left(X^p \right)^{\perp}$ and $(X^p)'$, which is the conclusion of our theorem because, thanks to Lemma~\ref{lemm1}, we have $\left(X^p \right)^{\perp}=(p)$. 

 For every $[w] \in M(\mathbb{R}_-)/\left(X^p \right)^{\perp} $, we have that $y  \in [w]$ if and only if $y=w+\psi$ with $\psi \in \left(X^p \right)^{\perp} $. Thus by definition of the orthogonal complement, the map $h$ is well defined because it does not depend on the choice of the represent $ w\in M(\mathbb{R}_-)$.

 The linear map $h$ is injective: reasoning by contradiction, there exists  $[w] \neq 0 \in M(\mathbb{R}_-)/\left(X^p \right)^{\perp}$ such that $h([w])=0$, i.e.,  $w \in \left(X^p \right)^{\perp}  $, which is a contradiction.

 We finally show now that the map $h$ is onto and it is an isometry. An element  $f \in (X^p)'$ is a continuous linear functional for the topology induced by the convergence on compact sets. By the Hahn-Banach extension theorem,
 we can extend $f$ on a continuous linear functional $\tilde{f}$ belonging to $(C_0(\mathbb{R}_+))'$, the dual space of $C_0(\mathbb{R}_+)$ with the duality product $\langle \cdot, \cdot \rangle_{C_0(\mathbb{R}_+)}$, such that
    \begin{equation}  
  \label{eq:jesaispas}
        \abs*{\langle \tilde{f},x \rangle_{C_0(\mathbb{R}_+)}} \le \|f \|_{\left(X^p \right)'}  \underset{t \in [0,T]}{\max} \abs*{x(t)},\quad  x \in C_0(\mathbb{R}_+).
        \end{equation}
        By the Riesz representation theorem, 
        there exists $\psi \in M(\mathbb{R}_-)$, with compact support included in $[-T,0]$ such that $\langle \tilde{f},x \rangle_{C_0(\mathbb{R}_+)}=(\psi,x)$ for all $x \in C_0(\mathbb{R}_+)$ and $\|\psi \|_{\rm TV}=\|f\|_{\left(X_p\right)'}$. Furthermore, for all $\phi \in \left(X^p \right)^{\perp}$, we have 
 \begin{equation}
 \begin{aligned}
     \|\psi+\phi \|_{\rm TV} =\underset{\substack{\|x\|_{\infty} \le 1, \\  x \in C_0(\mathbb{R}_+)}}         
        {\sup} \abs*{\left( \psi+\phi,x \right)}
        \ge \underset{\substack{\|x\|_{\infty} \le 1, \\  x \in X^p}}         
        {\sup} \abs*{\left( \psi+\phi,x \right)}
        = \underset{\substack{\|x\|_{[0,T]} \le 1, \\  x \in X^p}}         
        {\sup} \abs*{\left( \psi,x \right)}
        =\|\psi \|_{\rm TV} .
        \end{aligned}      
 \end{equation}
 Thus we have $\|[\psi]\|=\|\psi\|_{\rm TV}$. We deduce that $h([\psi])=f$ and $\| h([\psi]) \|_{\left(X^p\right)'}=\|f\|_{\left(X^p\right)'}=\| [\psi] \|$.
 
To sum up, we proved that the map $h$ is an isometric isomorphism between $M(\mathbb{R}_-)/\left(X^p \right)^{\perp}$ and  $(X^p)'$, achieving the proof of our theorem.

\end{proof}

\section{A corona theorem for a subalgebra of Radon measures negatively and compactly supported}

For the Banach algebra $M(\mathbb{R}_-)/(p)$, we call \emph{homomorphism} a continuous linear mapping $\phi:\,M(\mathbb{R}_-)/(p) \rightarrow \mathbb{C}$  satisfying
$\phi(FG)=\phi(F)\phi(G)$ for all $F,\,G \in M(\mathbb{R}_-)/(p)$. Recall that a character $\chi$ is  application from $\mathbb{R}_+$ to $\mathbb{C}$ such that $|\chi(t)|=1$ and $\chi(t+\tau)=\chi(t)\chi(\tau)$ for all $t,\, \tau \in \mathbb{R}_+$. We first give in Proposition~\ref{prop:class_homo} a description of the nonzero homomorphisms on $M(\mathbb{R}_-)/(p)$.

\begin{proposition}
\label{prop:class_homo}
    If $\phi \neq 0$ is a homomorphism in $M(\mathbb{R}_-)/(p)$ then either:
\begin{enumerate}[(1)]
  \item \label{item1} for every $h \in \Omega_{-}^{\mathrm{bd}}$ given by \eqref{eq:omega}:
    \begin{equation}
        \phi([h])=\begin{cases}
            h_j,\, \mbox{ if there is $\lambda_j =0$ and $h_j \neq 0$},\\
            0,\, \mbox{otherwise.}
        \end{cases}
    \end{equation}

\item \label{item2} or there exist $\sigma \in \mathbb{R}$ and a character $\chi$ such that, for every $h\in \Omega_{-}^{\mathrm{bd}}$ given by \eqref{eq:omega},
\begin{equation}\label{eq:item2}
    \phi([h])=\sum_{j=0}^{N} h_j e^{\sigma \lambda_j} \chi \left(\lambda_j\right),
\end{equation}
    \end{enumerate}
\end{proposition}

\begin{proof}
Let $\phi$ be a nonzero homomorphism $\phi \neq 0$ on $M(\mathbb{R}_-)/(p)$. In particular, by the continuity property, there exists $C>0$ (in fact $C$ can be taken equal to one because we are in a unital Banach algebra) such that:

\begin{equation}
\label{eq:seminorm_lemma}
    |\phi([h])| \le C  \|[h]\|,\quad \forall [h] \in M(\mathbb{R}_-)/(p).
\end{equation}

For $t\geq 0$, set $L(t)=|\phi([\delta_{-t}])|$, yielding a well-defined map from $\mathbb{R}_+$ to $\mathbb{R}_+$. 
We deduce from the equations~\eqref{def_norm} and \eqref{eq:seminorm_lemma} that $L$ is bounded over the interval $[0,T]$. Furthermore, from the property of homomorphisms, we deduce that $L$ is a multiplicative map, that is,
\begin{equation}
\label{eq_cauchy}
L(t_1+t_2)=L(t_1) L(t_2),\quad t_1,t_2 \in \mathbb{R}_+.
\end{equation}
Equation~\eqref{eq_cauchy} is a Cauchy equation of exponential type, see for instance \cite{kannappan2009functional}.
Since $\phi$ is a nonzero homomorphism, there exists $t_0 \geq 0$ such that $L(t_0)=c>0$ for some $t_0\in \mathbb{R}_+$. Thus we have $c=L(t_0)=L(t_0)L(0)=cL(0)$ and we deduce that $L(0)=1$. Following the discussion in \cite[Paragraph 1.5.1]{kannappan2009functional}, if there exists $t_*>0$ such that $L(t_*)=0$ then $L(t)=0$ for every $t>0$. In that case, $L$ is called the \emph{trivial solution} to the Cauchy equation of exponential type. Otherwise, the application of \cite[Theorem 1.37]{kannappan2009functional} gives the existence of $\sigma \in \mathbb{R}$ such that $L(t)=e^{\sigma t}$ for every $t\geq 0$. 
In summary, $L(0)=1$ and we have the following alternative:
\begin{enumerate}[(a)]
    \item either $L(t)=0$ for $t>0$ and then $\phi([\delta_{-t}])=0$ for $t>0$ and $\phi(\left[\delta_{0}\right])=1$, i.e., this corresponds to Item~\ref{item1} in the theorem with the help of Equation~\eqref{eq:omega};
    \item
    or there exists $\sigma \in \mathbb{R}$ such that $L(t)=e^{\sigma t}$ for $t\geq 0$, and then $\phi([\delta_{-t}])=e^{\sigma t} \chi(t)$ with $\chi(t)$ equal to $\phi([\delta_{-t}])e^{-\sigma t}$ which verifies $\abs{\chi(t)}=1$ for $t\geq 0$, i.e., $\chi$ is a character. According to Equation~\eqref{eq:omega}, one gets Item~\ref{item2} in the theorem.
    
\end{enumerate}

\end{proof}
We can now state and prove the corona theorem of this paper. 
\begin{theorem}
\label{thm_corona}
 Let  $K$ be a positive integer and $T$ be a strictly positive real number. Consider $f_i \in \Omega_{-}^{ T}$ for $i=1,\dots,K$. If there exists $\alpha>0$ such that
\begin{equation}
\label{eq:th1}
    \sum_{i=1}^{K} \abs*{\hat{f}_i(s)}\ge \alpha,\quad \forall s \in \mathbb{C},
\end{equation}
then there exist $g_i \in M(\mathbb{R}_{-})$ for $i=1,\dots,K$ satisfying
\begin{equation}
\label{eq:th2}
     \sum_{i=1}^K f_i* g_i=\delta_0.
\end{equation}
\end{theorem}

\begin{remark} 
    Condition~\eqref{eq:th1} is the same as that the condition of the corona theorem for $H^{\infty}$ by Carleson \cite{Carleson1962}. However, our corona theorem is much simpler because we worked in the algebra $M(\mathbb{R}_-)$ and we stated an interpolation result just for the elements belonging to $\Omega_{-}^{ T}$, for some $T>0$. More precisely, the properties of the homomorphisms given in Proposition~\ref{prop:class_homo} are harder to obtain for the algebra $H^{\infty}$. Furthermore, contrary to the corona theorem in $H^{\infty}$, we did not provide an estimate on the Laplace transform of the $g_i$ depending on $K$ and $\alpha$.
\end{remark}

\begin{proof}
Notice first that if $K=1$ then the conclusion holds trivially (since in that case $f_1=h_1\delta_{-\lambda_1}$ with $h_1\neq 0$) and we will assume then that $K\geq 2$ in the sequel. Moreover, one deduces from 
Condition~\eqref{eq:th1} that either every $f_i$ is  zero or a nonzero multiple of $\delta_0$ (and the result is again immediate or at least one of the $f_i$'s (let say $f_{K}$) has a nonempty support with a non zero element in its support. We will assume the latter in the sequel.

The first step of the proof consists in reducing the corona problem as stated in $M(\mathbb{R}_-)$ into a corona problem in the commutative unital Banach quotient algebra $A=M(\mathbb{R}_{-})/\left(f_{K}\right)$, where 
$\left(f_{K}\right)$, the two-sided ideal generated by $f_{K}$ over the commutative normed algebra
$\mathbb{R}_{-}$, is defined as
$\left\{f_{K}*h|\, h \in M(\mathbb{R}_-) \right\}$.
We note by $\left[\cdot\right]$ a class of equivalence of the quotient algebra $A$. Hence, we can interpret Equation~\eqref{eq:th2} as
\begin{equation}
\label{eq_red_quo_algebra}
\sum_{i=1}^{K-1} [f_i]* [g_i]=[\delta_0].
\end{equation}
Proving the theorem amounts to prove the existence of $[g_i] \in A $, $i=1,...,K-1$ satisfying \eqref{eq_red_quo_algebra}. Thanks to Theorem~\ref{th3}, $A$ is a commutative unital Banach algebra, and so we can use the Gelfand theory \cite[Chap VII, §8]{conway2019course}. Equation~\eqref{eq_red_quo_algebra} is equivalent to the fact that $[\delta_0]$ belongs to the two-sided ideal $\left([f_1],\cdots,[f_{K-1}]\right)$
generated by $[f_1],\cdots,[f_{K-1}]$ over the commutative algebra $A$ and defined as 
$\left\{[f_1]*[h_1]+\dotsc+[f_{K-1}]*[h_{K-1}]|\, [h_1],\dotsc,[h_{K-1}] \in A \right\}$. In other words,  Equation~\eqref{eq_red_quo_algebra} is equivalent to the fact that $\left([f_1],\cdots,[f_{K-1}]\right)$ is equal to $A$. 

Reasoning by contradiction, let us assume that 
$\left([f_1],\cdots,[f_{K-1}]\right)$ is not equal to $A$ and hence it is a proper ideal of $A$ which is, according to Krull's theorem (see for instance \cite[Theorem 11.3]{rudin1991functional}), included into a maximal ideal of $A$. In particular, $[f_1],...,[f_{K}]$ belong to a maximal ideal of $A$. 
The Gelfand representation theory states that the
maximal ideals are in bijection with the nonzero complex homomorphisms of $A$ so that 
a maximal ideal is included into the kernel of a unique nonzero homomorphism, see for instance \cite[Proposition 8.2, Chap. VII]{conway2019course}. Hence, 
there exists a nonzero homomorphism $\phi$ of $A$
for which 
\begin{equation}
\label{eq:proof_cor1before}
    \phi([f_1])=\phi([f_2])=\cdots=\phi([f_K])=0.
\end{equation}

If $\phi$ is given by Item~\ref{item1} of Proposition~\ref{prop:class_homo}, then 
the limit of the left-hand side of \eqref{eq:th1} tends to zero as $\Re(s)$ tends to 
$-\infty$, which contradicts Equation~\eqref{eq:th1}.
Assume now that $\phi$ is given Item~\ref{item2} of Proposition~\ref{prop:class_homo}. 
For every $k\in \llbracket 1,K\rrbracket$, the function $\hat{f}_k$ can be written as  
\begin{equation}\label{eq:proof_cor2}
   \hat{f}_k(s)= \sum_{l=0}^{n_k}f_{k,l} e^{s \lambda_{k,l}},\quad s \in \mathbb{C},
\end{equation}
where $n_k$ is an integer, $f_{k,l}$ a real number and $\lambda_{k,l} \in [0,T]$. We deduce from \eqref{eq:proof_cor1before}, \eqref{eq:proof_cor2} and \eqref{eq:item2} in Item~\ref{item2} of Proposition~\ref{prop:class_homo} that there exist $\sigma \in \mathbb{R}$ and a character $\chi$ such that 
\begin{equation}
\label{eq:proof_cor3}
    \sum_{l=0}^{n_k}f_{k,l}e^{\sigma \lambda_{k,l}} \chi( \lambda_{k,l})=0,\quad k\in \llbracket 1,K\rrbracket.
\end{equation}

We remark that, thanks to \cite[Proposition~3.9]{Chitour2016Stability},  there exist a positive integer $q$,  a rationally independent family $(r_1,\dots,r_q)$ of 
positive real numbers, and nonnegative integers $m_{k,l,j}$  for $l\in \llbracket 1,n_k\rrbracket$, $k\in \llbracket 1,K\rrbracket$ and $j\in \llbracket 1,q\rrbracket$ such that
\begin{equation}
\label{eq_diophantienne_real}
\lambda_{k,l} = \sum_{j=1}^q m_{k,l,j} r_j.
\end{equation}

Since $|\chi(t)|=1$ for all $t \in \mathbb{R}$, we have $\chi(r_j)=e^{2 \pi i \gamma_j}$ for some $\gamma_j \in \mathbb{R}$ and for $j=1,...,q$. It follows that

\begin{equation}
\label{eq:proof_cor4}
    \chi(\lambda_{k,l})=e^{2\pi i \sum\limits_{j=1}^q m_{k,l,j} \gamma_j},\quad l\in \llbracket 1,n_k\rrbracket, \quad k\in \llbracket 1,K\rrbracket.
\end{equation}

By the Kronecker approximation theorem (see e.g. \cite[Theorem 2, Chapter 2]{hlawka2012geometric}), for every $\epsilon>0$, there exist a real number $\beta$ and integers $p_1,...,p_q$ such that
\begin{equation}
\label{eq:proof_cor6}
   \abs*{\beta r_j-\gamma_j-p_j} \le \epsilon, \quad \mbox{for } j=1,...,q.
\end{equation}
From Equations~\eqref{eq_diophantienne_real}-\eqref{eq:proof_cor4}, we obtain for all $k=1,...,K$ and $l=1,...,n_k$
\begin{equation}
\label{eq:proof_cor6.1}
    \begin{aligned}
       \abs*{\chi(\lambda_{k,\,l})- e^{2 \pi i \beta \lambda_{k,\,l}}} &=\abs*{e^{2\pi i \sum\limits_{j=1}^q m_{k,l,j} \gamma_j}- e^{2 \pi i \beta  \sum\limits_{j=1}^q m_{k,l,j} r_j}}=
       \abs*{1-e^{2\pi i \sum\limits_{j=1}^qm_{k,l,j}(\gamma_j+p_j- \beta r_j)}}
    \end{aligned}  
\end{equation}
Using \eqref{eq:proof_cor6} in the above equation, one gets that there exists $C>0$
such that, for all $k=1,...,K$ and $l=1,...,n_k$, we have:
\begin{equation}
\label{eq:proof_cor6bis}
    \begin{aligned}
       \abs*{\chi(\lambda_{k,\,l})- e^{2 \pi i \beta \lambda_{k,\,l}}} \le C \epsilon .
    \end{aligned}  
\end{equation}

Let us define:
\begin{equation}
\label{eq:proof_cor5}
s_{\epsilon}=\sigma+i \beta \in \mathbb{C} \quad \mathrm{and} \quad \widetilde{C}=\underset{k\in \llbracket 1,K\rrbracket}{\sup} \sum_{l=0}^{n_k} \abs*{f_{k,l}}  .
\end{equation}
Hence, from equations~\eqref{eq:proof_cor3}-\eqref{eq:proof_cor6bis}-\eqref{eq:proof_cor5}, we get for all $k=1,...,K$:
\begin{equation}
\label{eq:proof_cor7}
\begin{aligned}
   \abs*{f_k(s_{\epsilon})}&=\abs*{f_k(s_{\epsilon})-\sum_{l=0}^{n_k}f_{k,l}e^{\sigma \lambda_{k,l}} \chi( \lambda_{k,l})},\\
    &=\abs*{\sum_{l=0}^{n_k}\left(f_{k,l}e^{\sigma \lambda_{k,l}} \chi( \lambda_{k,l})-f_{k,l}e^{\sigma \lambda_{k,l}} e^{2 \pi i \beta \lambda_{k,l}}  \right)},\\
&\le   \sum_{l=0}^{n_k} \abs*{f_{k,l}} \abs*{\chi( \lambda_{k,l})-e^{2 \pi i \beta \lambda_{k,l}}} ,\\
       &\le C \widetilde{C} \epsilon .
    \end{aligned}
\end{equation}
Letting $\epsilon$ tend to zero and using \eqref{eq:proof_cor7}, we build a sequence of complex numbers $\left(s_n\right)_{n\in \mathbb{N}}$ such that
\begin{equation}
\label{eq:proof_cor8}
    \lim\limits_{n \to +\infty}\hat{f}_1(s_n)=\cdots=\lim\limits_{n \to +\infty} \hat{f}_K(s_n)=0,
\end{equation}
which contradicts Equation~\eqref{eq:th1}. That completes the proof of Theorem~\ref{thm_corona}.

\end{proof}

\begin{remark}
Two questions remain open. Is it possible to find $g_1,...,g_K \in M(\mathbb{R}_-)$ (resp. $\Omega_-^{\rm bd}$), in the case where $f_1,...,f_K \in M(\mathbb{R}_-)$ (resp. $\Omega_-^{\rm bd}$), satisfying \eqref{eq:th2} if \eqref{eq:th1} holds?
Corona questions for measures 
can fail to have positive answers hold true as proved by the Wiener--Pitt phenomenon, see for instance \cite{nikolski1999search}. For $M(\mathbb{R}_-)$, the characterization of the nonzero homomorphisms of $M(\mathbb{R}_-)$ does not seem to be stated in the literature and therefore a corona theorem for this algebra is an open question.


\end{remark}



As application, we use Theorem~\ref{thm_corona} to establish a $L^1$-exact controllability of linear controlled delayed difference equations.



\section{$L^1$ exact controllability of linear difference delay control systems}
\label{sec:con}
The motivation to prove Theorem~\ref{thm_corona} arises from the study of the exact controllability problem of LCDDE. More precisely, let us consider a linear difference delay control system of the form
\begin{equation}
\label{system_lin_formel2}
 x(t)=\sum_{j=1}^NA_jx(t-\Lambda_j)+Bu(t) , \qquad t \ge 0,
\end{equation}
where, $d$ and $m$ are two integers, the state $x$ and the control $u$ belong 
to $\mathbb{R}^d$ and $\mathbb{R}^m$ respectively, and $A_1,\dotsc,A_N$ and $B$ are constant matrices with real entries of appropriate size. Without loss of generality, the delays $\Lambda_1, \dotsc, \Lambda_N$ are positive real numbers so that $\Lambda_1< \dotsb <\Lambda_N$. 

Since an LCDDE defines a infinite-dimensional dynamical system, we must introduce the functional spaces defining the state space and the control space of System~\eqref{system_lin_formel2}. If $I$ is a bounded interval of $\mathbb{R}$ and $n \in \mathbb{N}^*$, we note $L^1(I,\mathbb{R}^n)$ the space of integrable functions on $I$ with values in $\mathbb{R}^n$.

For every $\tilde{t} \geq 0$, $u \in L^1\left([0,\tilde{t}],\mathbb{R}^m\right)$, and $x_0 \in L^1\left([-\Lambda_N,0],\mathbb{R}^d\right)$, there exists a unique solution $x \in L^1\left([-\Lambda_N,\tilde{t}],\mathbb{R}^d\right)$ such that $x(\theta)=x_0(\theta)$ for almost all $\theta \in [-\Lambda_N,0]$ and $x(\cdot)$ satisfies Equation~\eqref{system_lin_formel2} for almost 
all $t \in [0,\tilde{t}]$, cf. \cite[Proposition 2.2]{Chitour2020Approximate}. 

We aim at reaching elements of $L^1\left([-\Lambda_N,0],\mathbb{R}^d\right)$ with an integrable control in a finite time along trajectories of \eqref{system_lin_formel2}. For that purpose, we introduce the following definition of exact controllability.

\begin{definition}\label{def:Lq-cont} System~\eqref{system_lin_formel2} is \emph{$L^1$ exactly controllable in time $T>0$} if 
for every  
$x_0,\phi \in L^1\left([-\Lambda_N,0],\mathbb{R}^d\right)$, there exists $u \in L^1\left([0,T],\mathbb{R}^m\right)$ such that 
the solution $x(\cdot)$ of System~\eqref{system_lin_formel2} starting at $x_0$ and associated with the control $u$ verifies
\begin{equation}
x(T+\theta)= \phi(\theta),\quad \mbox{\rm{for almost all}  $\theta \in \left[-\Lambda_N,0\right]$}.
\end{equation}
\end{definition}

In \cite{chitour:hal-03827918}, it is proved that the $L^1$ exact controllability of System~\eqref{system_lin_formel2} is equivalent to the resolution of a Bézout identity over the algebra of Radon measures compactly supported in $\mathbb{R}_-$, see \cite[Theorem 5.13]{chitour:hal-03827918}. This characterization allows one to give a necessary condition for the $L^1$ exact controllability but the remaining question whether this condition is also sufficient or not was left open in that reference. 
Using Theorem~\ref{thm_corona}, we bring a positive answer to this question. Since the $L^1$-controllability criterion is expressed in the frequency domain, we introduce the  matrix-valued holomorphic map 
\begin{equation}\label{eq:defH}
H(s):=I_d- \sum_{j=1}^N e^{- s \Lambda_j} A_j, \qquad s \in \mathbb{C},
\end{equation}
 where $I_d$ is the identity operator on $\mathbb{R}^d$. The matrix $H(\cdot)$ relates the control frequency with the state space frequency. More precisely, assuming that $u \in L^1\left(\mathbb{R},\mathbb{R}^m\right)$ and $u(t)=x(t)=0$ for $t <0$, we take the one--sided Laplace transform in \eqref{system_lin_formel2} and we obtain that there exists $\alpha>0$ such that:
 \begin{equation}
 \label{eq:defH1}
     H(s)X(s)=B U(s),\quad s \in \mathbb{C},\quad \Re(s)>\alpha,
 \end{equation}
 with 
 \begin{equation}
 \label{eq:defH2}
X(s)=\int_{-\infty}^{+\infty}x(t) e^{-st}dt \quad \mbox{and}    \quad U(s)=\int_{-\infty}^{+\infty}u(t) e^{-st}dt.
  \end{equation}
The existence of $\alpha>0$ such that Equation~\eqref{eq:defH1} is satisfied follows from classical exponential estimates for difference delay equations, see \cite[Chapter 9]{Hale}. 
We note $\overline{H(\mathbb{C})}$ the closure of the holomorphic matrix $H(\cdot)$ in the complex plane $\mathbb{C}$. The $d \times (d+m)$ matrix $\left[M,B\right]$ denotes the concatenation of a $d \times d$ matrix $M$ and the matrix $B$. Furthermore, $\rank\left[M,B\right]$ denotes the dimension of the range of the matrix $\left[M,B\right]$.

We state a sufficient and necessary criterion for the $L^1$ exact controllability for \eqref{system_lin_formel2} in the frequency domain.

\begin{theorem}
\label{main_result2} System~\eqref{system_lin_formel2} is $L^1$ exactly controllable in time $d \Lambda_N$ if and only if the two following conditions hold:
\begin{enumerate}[i)]
\item \label{assumption1-L1}$\rank\left[M,B\right]=d$ for every  $M\in \overline{H(\mathbb{C})}$,
\item \label{assumption2-L1} $\rank[A_N,B]=d$.
\end{enumerate}


\end{theorem}

\begin{proof}
Theorem~\ref{thm_corona} solves \cite[Conjecture~$5.18$]{chitour:hal-03827918} in the particular case where the $q_i$ belongs to $\Omega_{-}^{T}$ instead of $M(\mathbb{R}_-)$ for all $i=1,...,N$. Hence, Remark~$5.19$ and the discussion just below in the paper \cite{chitour:hal-03827918} allow us to conclude that \cite[Conjecture~$5.18$]{chitour:hal-03827918} is true for $q=1$, which is the result that we wanted. 

 
 
\end{proof}

 \section*{Acknowledgement: } 
The authors would like to thank N. Nikolski for discussions on corona problems for measures and Brett D. Wick for literature references on corona theorems for almost periodic functions. 

\bibliographystyle{abbrv}


\bibliography{ifacconf}

\end{document}